\documentclass[11pt]{article}
\usepackage[T1]{fontenc}
\usepackage{lmodern}
\usepackage[a4paper,margin=2.5cm]{geometry}
\usepackage{amsmath,amssymb,amsthm}
\usepackage{graphicx}
\usepackage{booktabs}
\usepackage{array}
\usepackage{subcaption}
\usepackage{xcolor}
\usepackage{microtype}
\usepackage{enumitem}
\setlist[enumerate]{itemsep=1pt,parsep=0pt,topsep=3pt}
\usepackage[hidelinks]{hyperref}
\usepackage{authblk}

\newtheorem{theorem}{Theorem}
\newtheorem{proposition}[theorem]{Proposition}

\theoremstyle{definition}

\theoremstyle{remark}

\newcommand{\cP}{\mathcal{P}}
\newcommand{\cM}{\mathcal{M}}

\title{A second eight-faced polyhedron\\ in which every two faces share an edge}

\author{Gergely R\"ost, Viktor V\'igh \\ \small\texttt{rost@math.u-szeged.hu, vigvik@math.u-szeged.hu}}
\affil{Bolyai Institute, University of Szeged, Hungary}\date{26 September 2026}

\begin{document}
\maketitle

\begin{abstract}
We report a novel polyhedral surface of genus~3 embedded in $\mathbb{R}^3$ with eight planar, simple, non-convex nonagonal faces, 24 vertices and 36 edges, in which every two faces share at least one edge: 20 pairs of faces share one edge and 8 pairs share two collinear edges. Its face planes are $3x-4y-2z=5$, $-2x+5y-5z=3$ and their images under the half-turns about the three coordinate axes, and all vertices are rational. The polyhedron has the same face vector, face sizes and number of edge multiplicities as the polyhedron described by Mizhaev, but it is not combinatorially equivalent to it. Our realisation has the symmetry group $D_2$ of order~4, whereas Mizhaev's polyhedron has a rotoreflection symmetry. The example was found by a computational geometric search, and all its properties were verified in exact rational arithmetic.
\end{abstract}

\section{Introduction}

In 1977, Lajos Szilassi found a polyhedron with seven hexagonal faces, each of which shares an edge with each of the other six~\cite{Szilassi86,Gardner78}. It is a torus whose combinatorial structure is dual to that of the Cs\'asz\'ar polyhedron, a seven-vertex torus in which every two vertices are joined by an edge~\cite{Csaszar49}. The tetrahedron and the Szilassi polyhedron have long been referred to as the only known polyhedra in which each face shares an edge with every other face; see for example~\cite{Arseneva24}. We call such polyhedra face-neighbourly.

A simple count explains why such examples are rare. Suppose a polyhedron with $n$ faces having three faces at every vertex, and every two faces share exactly one edge. Then $E=\binom n2$, $3V=2E$, and Euler's formula $V-E+F=2-2g$ gives
\begin{equation}\label{eq:genus}
  g=\frac{(n-3)(n-4)}{12}.
\end{equation}
This is an integer only for $n\equiv 0,3,4,7 \pmod{12}$, so under these assumptions the next case after the tetrahedron ($n=4$) and the Szilassi polyhedron ($n=7$) has $n=12$ faces, 66 edges, 44 vertices and genus~6. Its combinatorial type would be dual to one of the 59 neighbourly triangulations of the orientable surface of genus~6 with 12 vertices~\cite{Altshuler96}, and whether it can be realised with flat faces is open. The dual question, whether one of those 59 triangulations can be realised as a polyhedron, was settled negatively by Schewe using oriented matroids and satisfiability solvers~\cite{Schewe10}. His result does not transfer to the face version, because the dual of an embedded polyhedron need not be embedded.

The question of every two faces sharing at least one edge, allows other face counts however, because some pairs of faces may meet along two collinear edges. Mizhaev~\cite{Mizhaev20,Mizhaev26} constructed such a polyhedron with eight nonagonal faces and genus~3, in which 20 pairs of faces share one edge and 8 pairs share two. In this note we present a second polyhedron with exactly these numbers which is not combinatorially equivalent to Mizhaev's. The two differ in how the doubled pairs are arranged, in their combinatorial automorphism groups, and in their symmetry and chirality.

\section{The polyhedron}

Let $R_x,R_y,R_z$ denote the half-turns about the coordinate axes, e.g.\ $R_x(x,y,z)=(x,-y,-z)$. Together with the identity they form the group $D_2\cong\mathbb{Z}_2\times\mathbb{Z}_2$. The polyhedron $\cP$ has two orbits of four faces under $D_2$,
\[
  A_1:\ 3x-4y-2z=5, \qquad B_1:\ -2x+5y-5z=3,
\]
with $A_2=R_z(A_1)$, $A_3=R_x(A_1)$, $A_4=R_y(A_1)$ and similarly for $B_2,B_3,B_4$. The eight planes are listed in Table~\ref{tab:planes}. They are in general position: no two are parallel, and every three of them meet in a single point that lies on no fourth plane. Each vertex of $\cP$ is the intersection point of the planes of the three faces containing it, so all coordinates are rational. Table~\ref{tab:vertices} lists the 24 vertices together with the faces meeting at each. The vertices form six $D_2$-orbits of size four, numbered consecutively; within each orbit the order is $v,\ R_z v,\ R_x v,\ R_y v$.

\begin{theorem}\label{thm:main}
The eight polygons of Table~\ref{tab:planes} form a polyhedron $\cP$ of genus~$3$ with the following properties.
\begin{enumerate}
\item $\cP$ has $8$ faces, $36$ edges and $24$ vertices; every face is a simple non-convex nonagon and every vertex lies in exactly three faces, so $\cP$ is equivelar of type $\{9,3\}$.
\item Every two faces share an edge. The pairs sharing two (collinear) edges are
\[
  A_1B_2,\ A_1B_4,\ A_3B_2,\ A_3B_4 \quad\text{and}\quad A_2B_1,\ A_2B_3,\ A_4B_1,\ A_4B_3,
\]
and the remaining $20$ pairs share exactly one edge.
\item The faces $A_i$ have three reflex corners each and the faces $B_i$ have two.
\item The symmetry group of $\cP$ is exactly $D_2=\{\mathrm{id},R_x,R_y,R_z\}$. In particular $\cP$ is chiral.
\item The combinatorial automorphism group of $\cP$ is the dihedral group of order $8$. It acts simply transitively on the eight faces, and all its elements preserve orientation.
\item The volume of $\cP$ is $351105133/18170460\approx 19.323$, and its convex hull is the tetrahedron spanned by vertices $1$--$4$.
\end{enumerate}
\end{theorem}

\begin{table}[thbp]
\centering
\caption{Face planes and faces of $\cP$. Each face is given by its vertices in counter-clockwise order as seen from outside the solid.}
\label{tab:planes}
\begin{tabular}{llll}
\toprule
face & plane & image of & boundary cycle \\
\midrule
$A_1$ & $3x - 4y - 2z = 5$ & --- & 1 2 5 7 17 13 21 12 9 \\
$A_2$ & $-3x + 4y - 2z = 5$ & $R_z(A_1)$ & 6 8 18 14 22 11 10 2 1 \\
$A_3$ & $3x + 4y + 2z = 5$ & $R_x(A_1)$ & 7 5 19 15 23 10 11 3 4 \\
$A_4$ & $-3x - 4y + 2z = 5$ & $R_y(A_1)$ & 4 3 8 6 20 16 24 9 12 \\
$B_1$ & $-2x + 5y - 5z = 3$ & --- & 1 9 24 23 15 14 18 20 6 \\
$B_2$ & $2x - 5y - 5z = 3$ & $R_z(B_1)$ & 13 17 19 5 2 10 23 24 16 \\
$B_3$ & $-2x - 5y + 5z = 3$ & $R_x(B_1)$ & 21 13 16 20 18 8 3 11 22 \\
$B_4$ & $2x + 5y + 5z = 3$ & $R_y(B_1)$ & 4 12 21 22 14 15 19 17 7 \\
\bottomrule
\end{tabular}
\end{table}

\begin{table}[tbp]
\centering
\caption{The 24 vertices of $\cP$ and the three faces meeting at each. Rows are grouped by $D_2$-orbits.}
\label{tab:vertices}
\small
\begin{tabular}{rrrrl c rrrrl}
\toprule
\# & $x$ & $y$ & $z$ & faces & & \# & $x$ & $y$ & $z$ & faces\\
\midrule
1 & $-38/7$ & $-57/14$ & $-5/2$ & $A_1\,A_2\,B_1$ & & 13 & $13/11$ & $-3/5$ & $26/55$ & $A_1\,B_2\,B_3$ \\
2 & $38/7$ & $57/14$ & $-5/2$ & $A_1\,A_2\,B_2$ & & 14 & $-13/11$ & $3/5$ & $26/55$ & $A_2\,B_1\,B_4$ \\
3 & $-38/7$ & $57/14$ & $5/2$ & $A_3\,A_4\,B_3$ & & 15 & $13/11$ & $3/5$ & $-26/55$ & $A_3\,B_1\,B_4$ \\
4 & $38/7$ & $-57/14$ & $5/2$ & $A_3\,A_4\,B_4$ & & 16 & $-13/11$ & $-3/5$ & $-26/55$ & $A_4\,B_2\,B_3$ \\ \addlinespace[2pt]
5 & $5/3$ & $-1/15$ & $2/15$ & $A_1\,A_3\,B_2$ & & 17 & $3/2$ & $-1/4$ & $1/4$ & $A_1\,B_2\,B_4$ \\
6 & $-5/3$ & $1/15$ & $2/15$ & $A_2\,A_4\,B_1$ & & 18 & $-3/2$ & $1/4$ & $1/4$ & $A_2\,B_1\,B_3$ \\
7 & $5/3$ & $1/15$ & $-2/15$ & $A_1\,A_3\,B_4$ & & 19 & $3/2$ & $1/4$ & $-1/4$ & $A_3\,B_2\,B_4$ \\
8 & $-5/3$ & $-1/15$ & $-2/15$ & $A_2\,A_4\,B_3$ & & 20 & $-3/2$ & $-1/4$ & $-1/4$ & $A_4\,B_1\,B_3$ \\ \addlinespace[2pt]
9 & $-37/38$ & $-5/4$ & $-111/76$ & $A_1\,A_4\,B_1$ & & 21 & $31/23$ & $-62/115$ & $3/5$ & $A_1\,B_3\,B_4$ \\
10 & $37/38$ & $5/4$ & $-111/76$ & $A_2\,A_3\,B_2$ & & 22 & $-31/23$ & $62/115$ & $3/5$ & $A_2\,B_3\,B_4$ \\
11 & $-37/38$ & $5/4$ & $111/76$ & $A_2\,A_3\,B_3$ & & 23 & $31/23$ & $62/115$ & $-3/5$ & $A_3\,B_1\,B_2$ \\
12 & $37/38$ & $-5/4$ & $111/76$ & $A_1\,A_4\,B_4$ & & 24 & $-31/23$ & $-62/115$ & $-3/5$ & $A_4\,B_1\,B_2$ \\
\bottomrule
\end{tabular}
\end{table}

\begin{figure}[tbp]
\centering
\begin{subfigure}{0.43\textwidth}\centering\includegraphics[width=\textwidth]{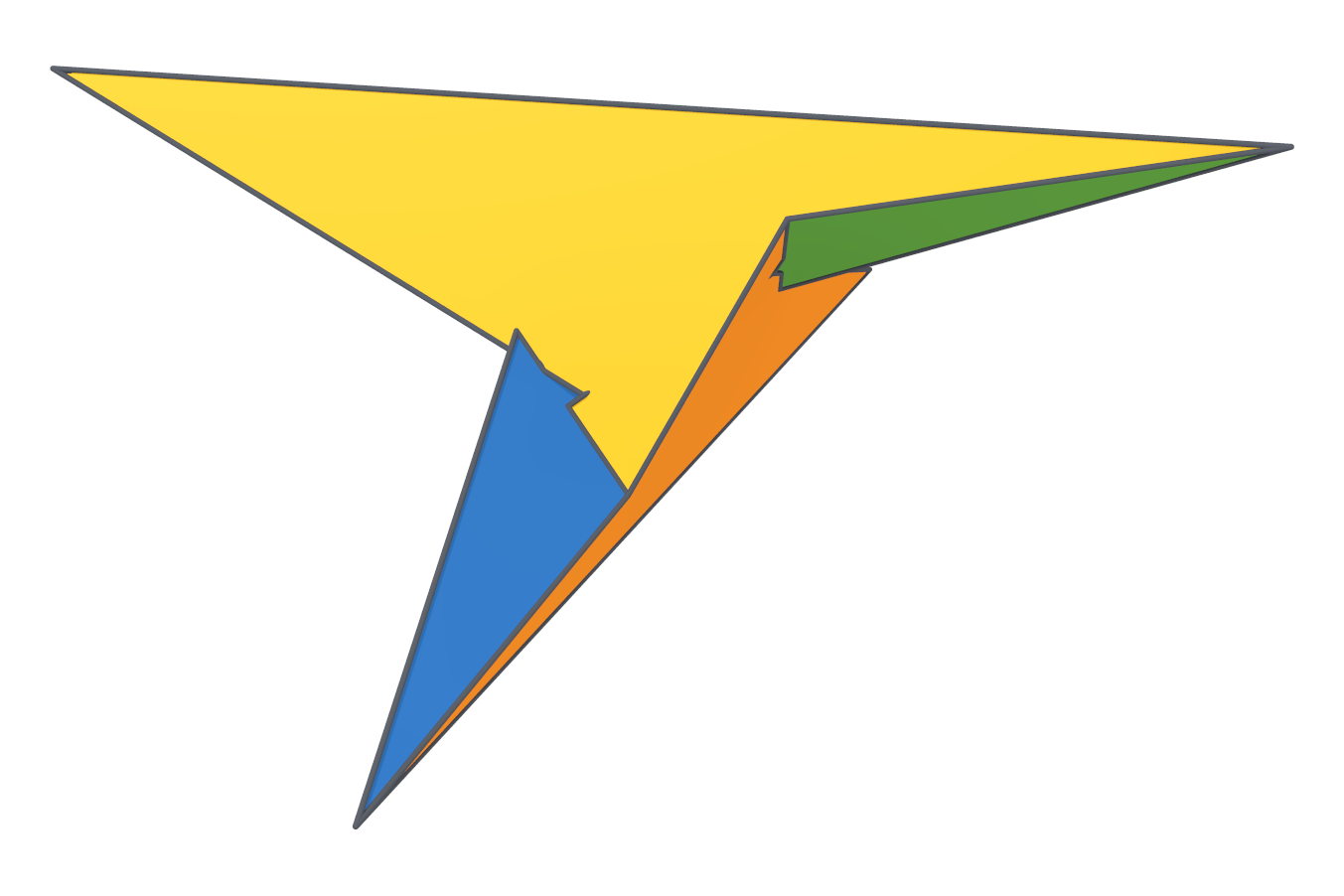}\caption{}\end{subfigure}\hspace{0.04\textwidth}
\begin{subfigure}{0.43\textwidth}\centering\includegraphics[width=\textwidth]{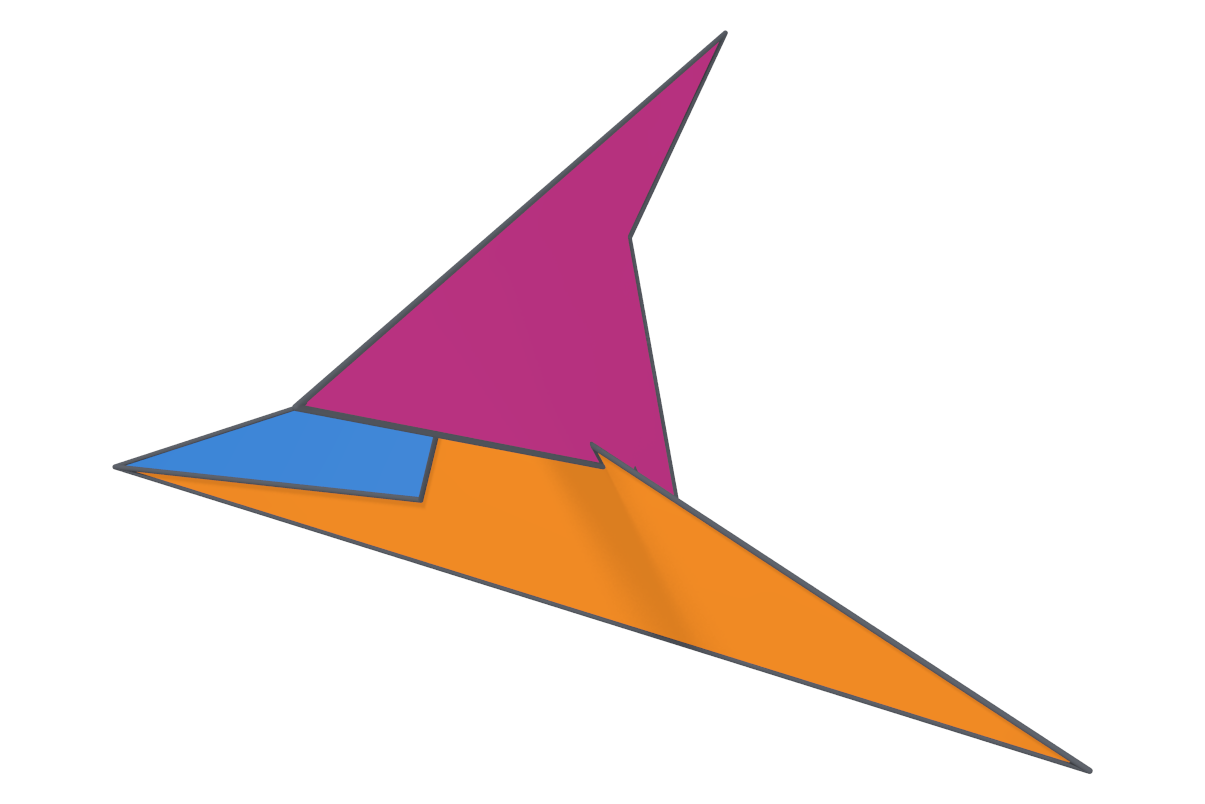}\caption{}\end{subfigure}\\[4pt]
\begin{subfigure}{0.43\textwidth}\centering\includegraphics[width=\textwidth]{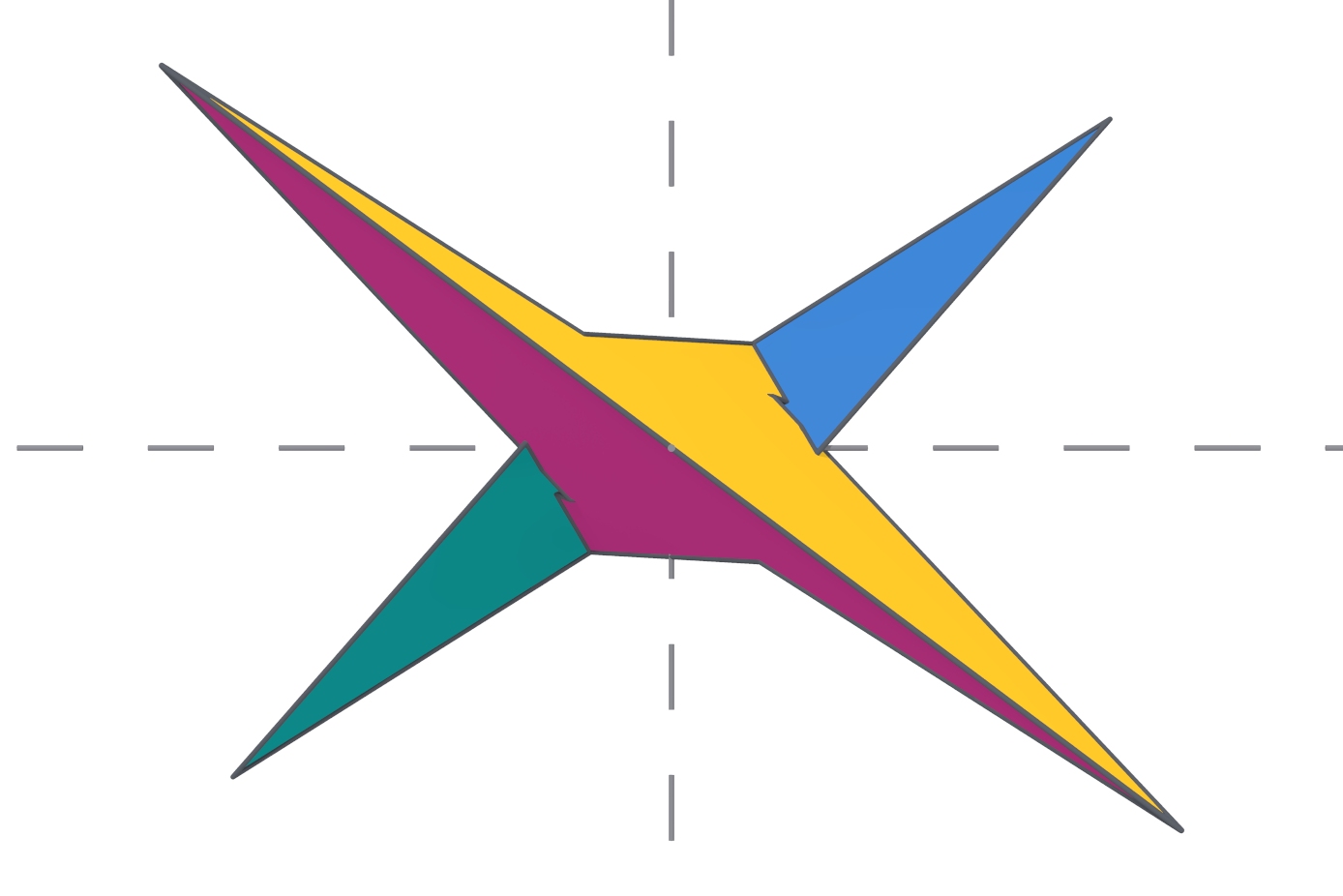}\caption{}\end{subfigure}\hspace{0.04\textwidth}
\begin{subfigure}{0.43\textwidth}\centering\includegraphics[width=\textwidth]{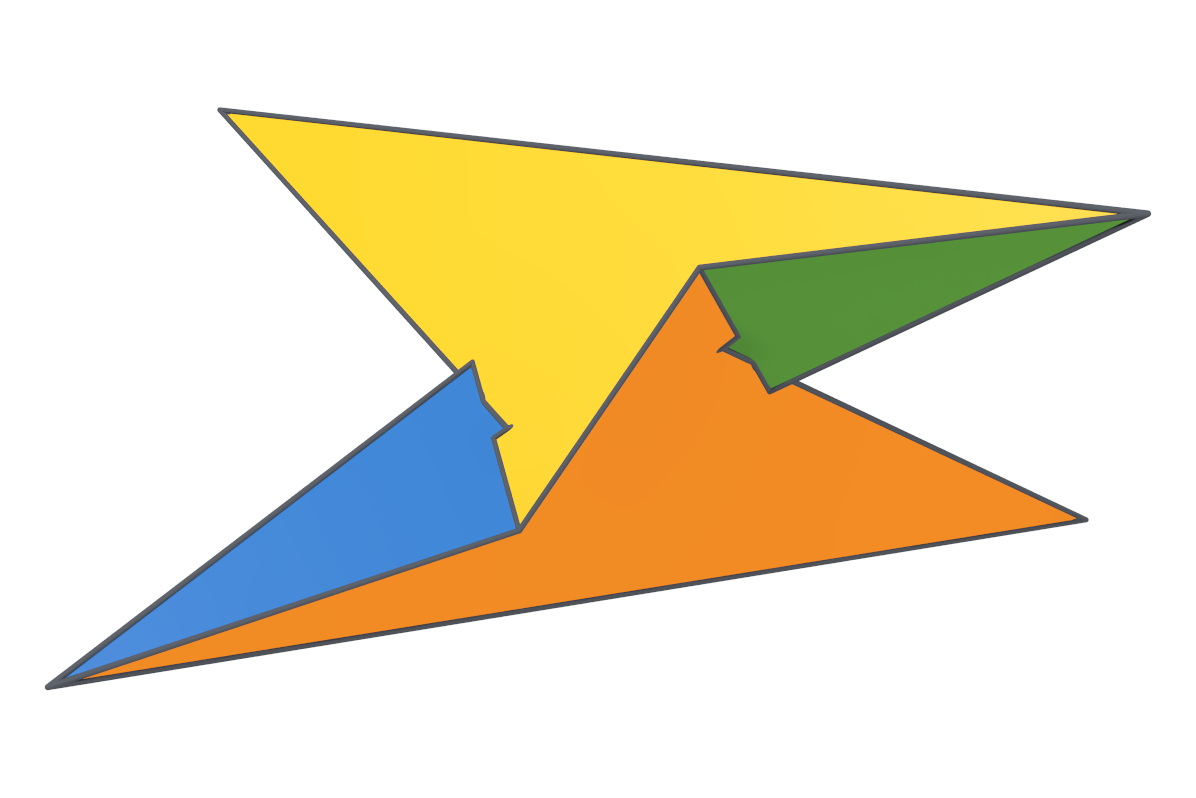}\caption{}\end{subfigure}
\caption{The polyhedron $\cP$ seen from four directions. Faces $A_1,\dots,A_4$ are drawn in warm colours (red, orange, yellow, magenta) and $B_1,\dots,B_4$ in cool colours (teal, blue, green, violet), as in the other figures. In (c) the view is along the $z$-axis; the dashed lines are the $x$- and $y$-axes, two of the three half-turn axes. The four long spikes end at the vertices $1$--$4$, which span the convex hull of~$\cP$.}
\label{fig:views}
\end{figure}

Items 1--3, 5 and~6 are verified in exact arithmetic as described in Section~\ref{sec:verification}; for item~4 we also give a short proof.

\begin{proof}[Proof of item 4]
Every symmetry $g$ of $\cP$ permutes the vertices, so it fixes their centroid, which is the origin because the vertex set is invariant under $D_2$. Hence $g$ is an orthogonal linear map that permutes the face planes. The planes $A_i$ are at distance $5/\sqrt{29}$ from the origin and the planes $B_i$ at distance $3/\sqrt{54}=1/\sqrt6$, so $g$ permutes the planes $A_i$. Write $A_i$ as $\langle a_i,x\rangle=5$ with $a_1=(3,-4,-2)$, $a_2=(-3,4,-2)$, $a_3=(3,4,2)$ and $a_4=(-3,-4,2)$. Then $g(A_i)=A_j$ means $g a_i=a_j$, so $g$ permutes the vectors $a_i$ and preserves their inner products
\[
  \langle a_1,a_2\rangle=\langle a_3,a_4\rangle=-21,\qquad \langle a_1,a_3\rangle=\langle a_2,a_4\rangle=-11,\qquad \langle a_1,a_4\rangle=\langle a_2,a_3\rangle=3 .
\]
The only permutations of $\{a_1,\dots,a_4\}$ preserving these values are the identity and the three double transpositions. Since $a_1,a_2,a_3$ span $\mathbb{R}^3$, $g$ is determined by its permutation, and each of the four possibilities is realised by an element of $D_2$. So the symmetry group is contained in $D_2$; conversely, $R_x$, $R_y$ and $R_z$ are symmetries by construction (Table~\ref{tab:planes}; see also Section~\ref{sec:verification}). Hence the symmetry group is $D_2$, which contains no orientation-reversing isometry.
\end{proof}

By item~5 the eight faces are alike combinatorially, although geometrically the faces
$A_i$ and $B_i$ differ, e.g.\ in the number of reflex corners. As no automorphism
reverses orientation, every polyhedron combinatorially equivalent to $\mathcal P$
is chiral: an improper isometry mapping a polyhedron onto itself maps outward normals to
outward normals and reverses the orientation of space, hence that of the surface.
\begin{figure}[tbp]
\centering
\includegraphics[width=0.82\textwidth]{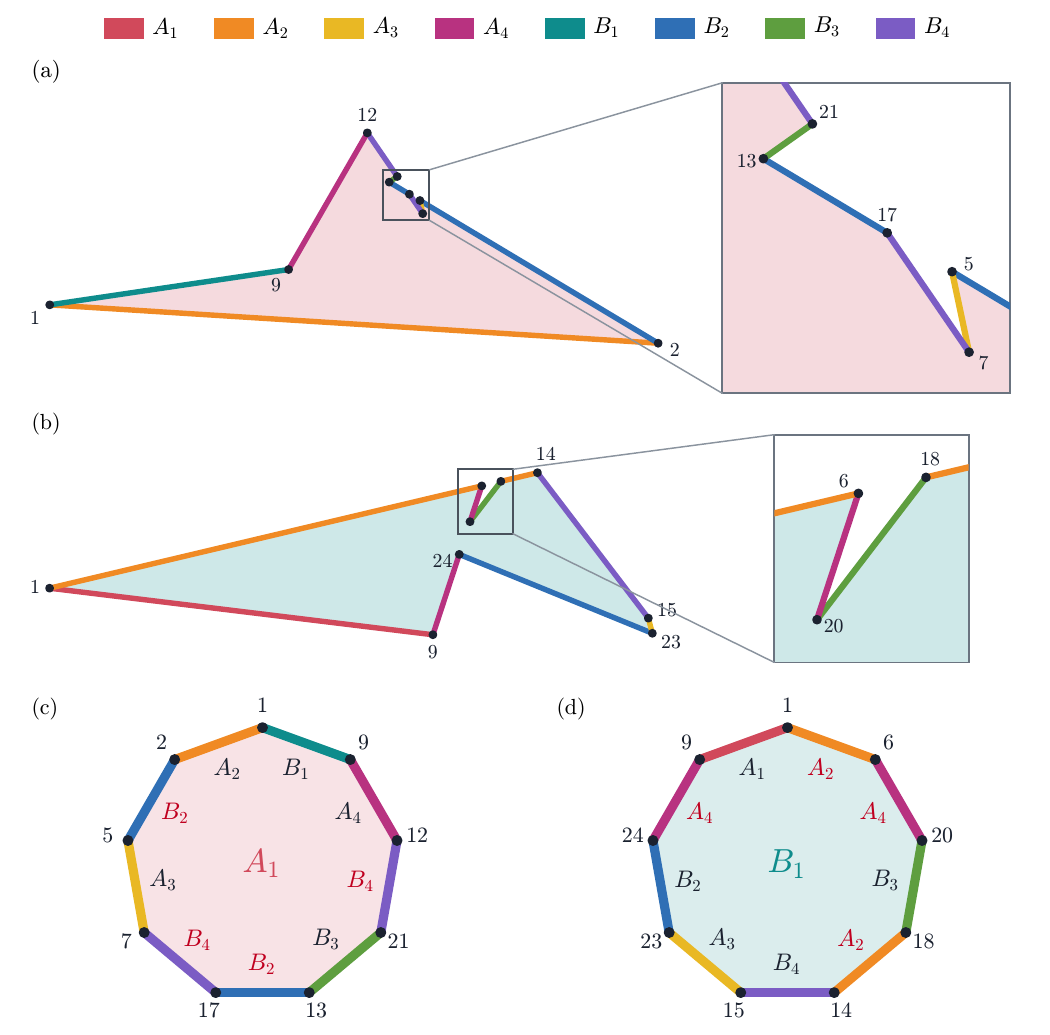}
\caption{(a), (b) The faces $A_1$ and $B_1$ in their true shapes, seen from outside, with every edge coloured by the face on the other side of it (colour key at the top); the insets magnify the small rectangles. The other faces are congruent to these, by the symmetries. The two edges that $A_1$ shares with $B_2$ (blue) lie on one line, and so do the two it shares with $B_4$ (violet). (c), (d) The same boundary cycles drawn as regular nonagons, with the vertex numbers of Table~\ref{tab:vertices} and the neighbouring face written along each edge. Red labels mark the two faces that are met twice.}
\label{fig:faces}
\end{figure}

\section{Verification}\label{sec:verification}

The following verifications have been carried out.
\begin{enumerate}
\item \emph{Planes and vertices.} No two planes are parallel, every three meet in one point, and no four pass through a common point. Every vertex lies on the planes of the three faces containing it and on no other face plane.
\item \emph{Simplicity and reflex corners.} In each face no two non-consecutive edges
meet, and no two consecutive edges overlap. Seen from outside, the boundary of $A_1$
turns clockwise exactly at the vertices $7$, $9$, $13$, and that of $B_1$ at $20$, $24$;
the other faces are their images under $D_2$. So every face is non-convex, with three
reflex corners in each $A_i$ and two in each $B_i$.
\item \emph{Closedness and orientability.} Every edge occurs in exactly two face cycles, once in each direction.
\item \emph{Vertex links.} At every vertex the faces around it form a single cycle, so the surface is a closed 2-manifold. Its Euler characteristic is $24-36+8=-4$, hence its genus is~$3$.
\item \emph{Embedding.} For each of the $28$ pairs of faces the intersection of the two closed polygons is computed on the line where their planes meet, as a union of closed intervals. In every case it equals the union of the edges and vertices the two faces share.
\item \emph{Orientation and volume.} The signed volume computed from the oriented faces is positive, so all face cycles are oriented outwards; its value is given in Theorem~\ref{thm:main}. All vertices lie in the tetrahedron spanned by vertices $1$--$4$.
\item \emph{Symmetries.} The program finds all orthogonal maps that permute the eight face planes, and keeps those that map vertices to vertices and face cycles to face cycles. For $\cP$ these are exactly the identity, $R_x$, $R_y$ and $R_z$.
\item \emph{Automorphisms.} An automorphism of the map is determined by the image of one directed edge, so the program tries all $72$ images, for the map and for its mirror image. For $\cP$ it finds eight automorphisms, all orientation-preserving: the identity, five involutions and two elements of order~$4$, so the group is dihedral (no other group of order~8 has five involutions). No automorphism except the identity fixes a face.
\end{enumerate}
Items 1--5 show that $\cP$ is an embedded closed orientable surface of genus 3 with the stated faces, edges and vertices; the edge lists give the multiplicities in Theorem~\ref{thm:main}. 

\section{Comparison with Mizhaev's polyhedron}

Mizhaev's polyhedron $\cM$~\cite{Mizhaev20,Mizhaev26} has the face planes
\begin{align*}
F_{1,2}&:\ 21x+3y\mp 10z=\mp 1440, & F_{3,4}&:\ 3x\pm z=\pm 234,\\
F_{7,8}&:\ 3x-21y\pm 10z=\mp 1440, & F_{5,6}&:\ 3y\mp z=\pm 234,
\end{align*}
with the upper signs for the face with the smaller index. Applying our construction procedure to these planes reproduces his polyhedron, with integer vertices in $[-300,300]\times[-300,300]\times[-234,234]$. Its symmetry group is the cyclic group of order~4 generated by the rotoreflection $T(x,y,z)=(y,-x,-z)$~\cite{Mizhaev26}, which permutes the faces as $F_1\to F_7\to F_2\to F_8\to F_1$ and $F_3\to F_6\to F_4\to F_5\to F_3$. Every combinatorial automorphism of $\cM$ is induced by a power of $T$~\cite{Mizhaev26}, so its automorphism group coincides with its symmetry group and is generated by an orientation-reversing element. The program of Section~\ref{sec:verification} confirms all of this. Thus $\cM$ realises all its combinatorial symmetries, while $\cP$ realises half of them. Table~\ref{tab:compare} compares the two polyhedra.

\begin{table}[tbp]
\centering
\caption{Comparison of $\cP$ with Mizhaev's polyhedron $\cM$.}
\label{tab:compare}
\begin{tabular}{@{}lll@{}}
\toprule
 & $\cP$ (this note) & $\cM$ \cite{Mizhaev26} \\
\midrule
faces, edges, vertices, genus & $8,36,24,3$ & $8,36,24,3$\\
face size, vertex degree & $9,3$ & $9,3$\\
pairs sharing one / two edges & $20$ / $8$ & $20$ / $8$\\
faces with $3$ / $2$ reflex corners & $4$ / $4$ & $4$ / $4$\\
\addlinespace
graph of doubled pairs & two $4$-cycles & one $8$-cycle\\
combinatorial automorphisms & dihedral, order $8$ & cyclic, order $4$\\
\quad of which orientation-reversing & none & two\\
symmetry group & $D_2$ (three half-turns) & $\langle T\rangle$, $T$ a rotoreflection\\
chirality & chiral (every realisation) & achiral\\
\addlinespace
vertex coordinates & rational & integers\\
largest plane coefficient & $5$ & $1440$\\
\bottomrule
\end{tabular}
\end{table}

\begin{proposition}\label{prop:distinct}
$\cP$ and $\cM$ are not combinatorially equivalent, even if orientation reversal is allowed.
\end{proposition}

\begin{proof}
For a face-neighbourly polyhedron $X$ let $D(X)$ be the graph whose vertices are the faces of $X$, with two faces joined when they share two edges. Any combinatorial equivalence $X\to Y$ induces an isomorphism $D(X)\to D(Y)$. By Theorem~\ref{thm:main}, $D(\cP)$ consists of the two $4$-cycles $A_1B_2A_3B_4$ and $A_2B_1A_4B_3$. For $\cM$ the doubled pairs are $F_1F_4$, $F_1F_5$, $F_2F_3$, $F_2F_6$, $F_3F_7$, $F_4F_8$, $F_5F_7$ and $F_6F_8$, which form the single $8$-cycle $F_1F_4F_8F_6F_2F_3F_7F_5$. The graphs are not isomorphic.
\end{proof}

\begin{figure}[tbp]
\centering
\includegraphics[width=\textwidth]{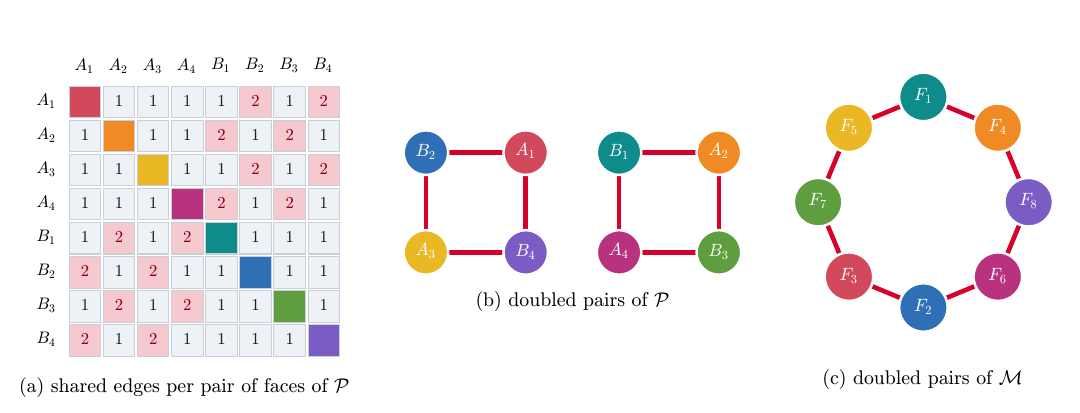}
\caption{(a) Number of edges shared by each pair of faces of $\cP$. (b) The doubled pairs of $\cP$ form two $4$-cycles. (c) The doubled pairs of Mizhaev's polyhedron form one $8$-cycle; its faces $F_3,\dots,F_6$, which have three reflex corners each, are drawn in warm colours.}
\label{fig:graphs}
\end{figure}

\begin{figure}[tbp]
\centering
\begin{subfigure}{0.36\textwidth}\centering\includegraphics[width=\textwidth]{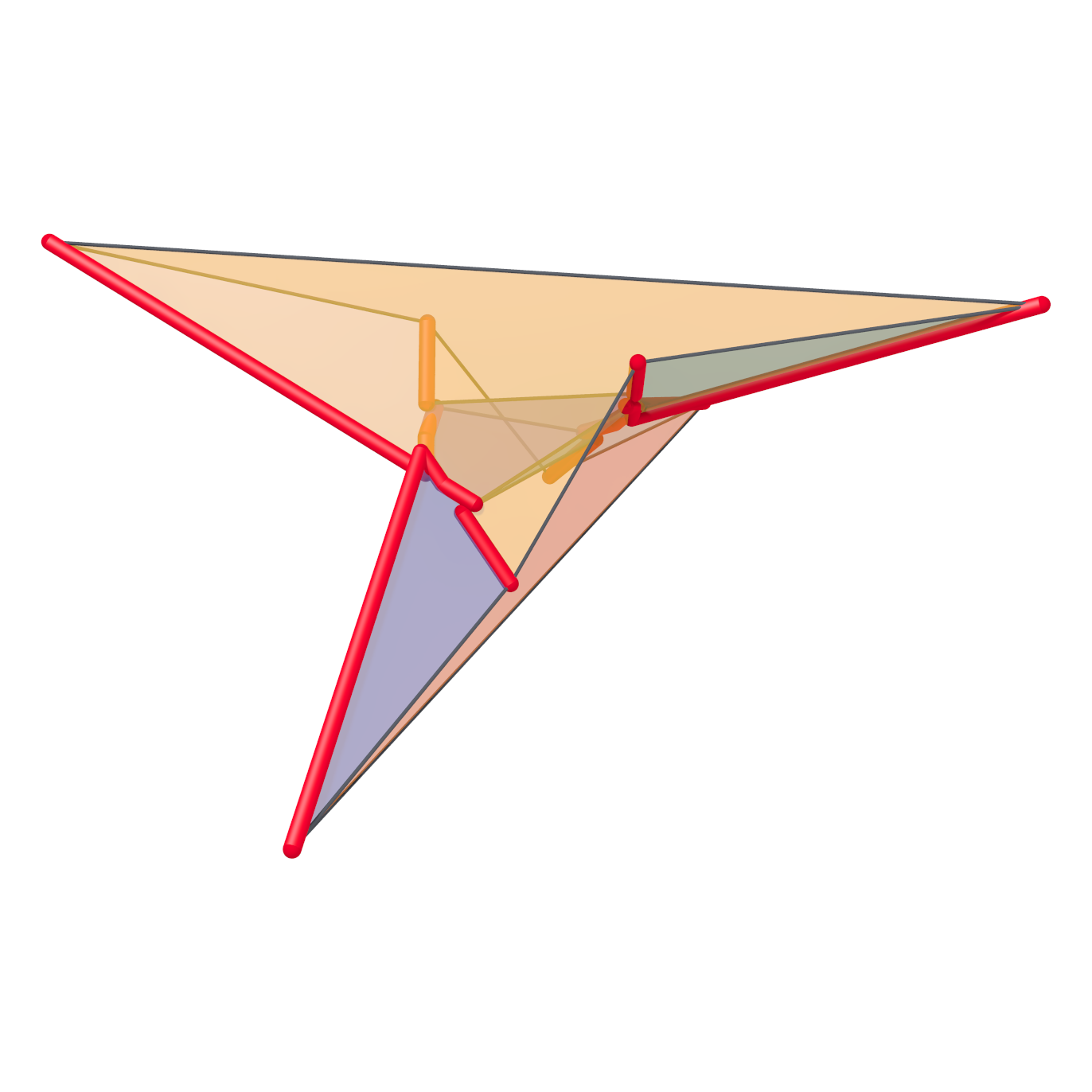}\caption{$\cP$}\end{subfigure}\hspace{0.08\textwidth}
\begin{subfigure}{0.36\textwidth}\centering\includegraphics[width=\textwidth]{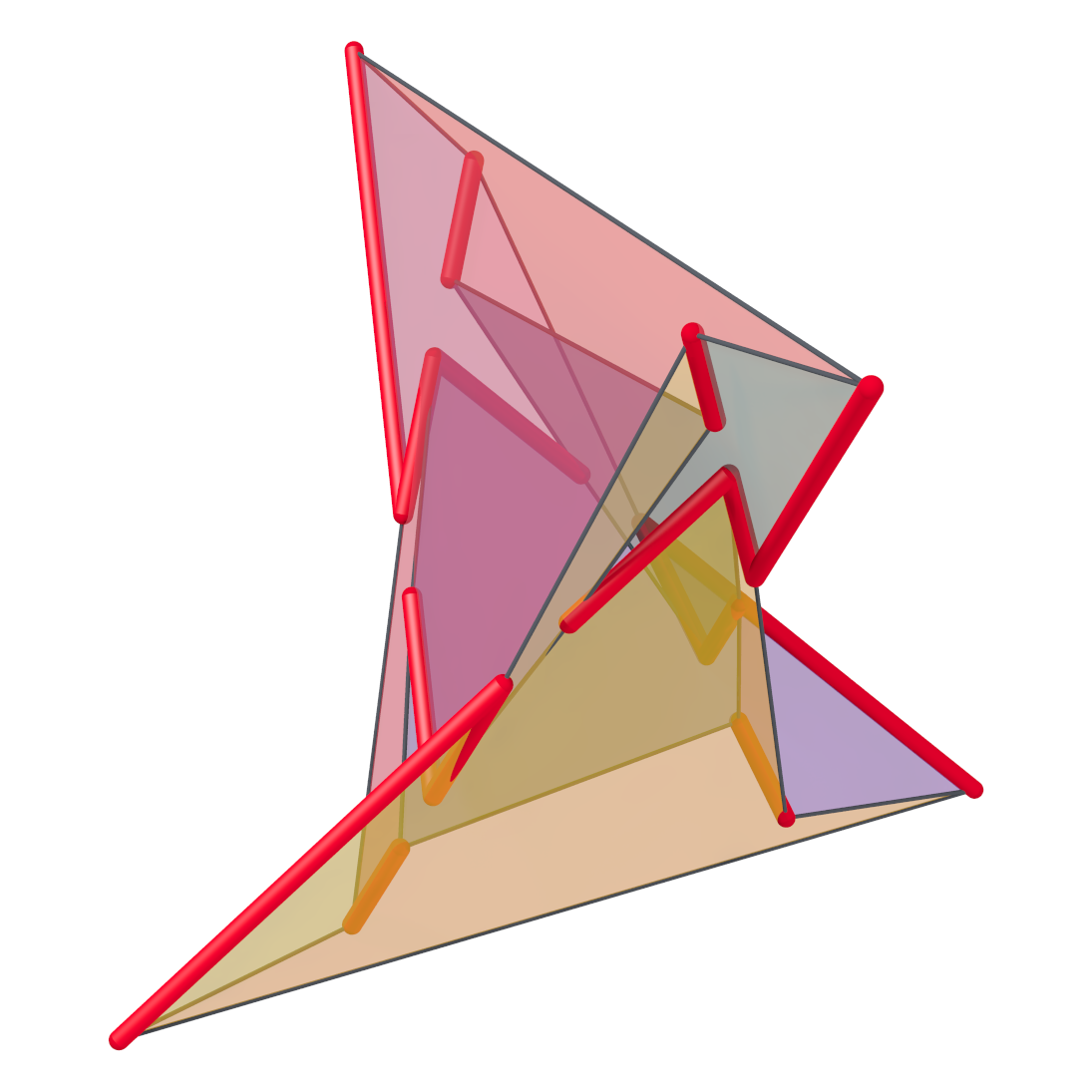}\caption{$\cM$}\end{subfigure}
\caption{The two polyhedra with translucent faces. The sixteen edges belonging to doubled pairs are drawn in red; both panels use the same viewing direction. In $\cP$ several doubled edges run along the long spikes, while in $\cM$ they are spread around the central part of the surface.}
\label{fig:compare}
\end{figure}

\section{Concluding remarks}

The example was found by a computational geometric search, using significant computational resources. The search algorithm was developed, optimized, implemented in Python, and verified through a highly autonomous AI framework. The method rediscovered the Szilassi polyhedron and Mizhaev's example as well. The goal of this note is solely to report the discovery of our new polyhedron. A detailed description of the computational methodology and its further applications will be presented in a subsequent work.

Apart from the two eight-faced examples, the only known face-neighbourly polyhedra are still the tetrahedron and the Szilassi polyhedron. It is not known if any further examples exist. In particular, the case of twelve faces, each pair sharing exactly one edge, remains open as well.

\paragraph{Acknowledgements.} The authors acknowledge the significant use of artificial intelligence (Claude) in most phases of this work. The authors reviewed and validated all of its outputs.

\end{document}